\documentclass{article}

\usepackage[english]{babel}

\usepackage[letterpaper,top=2cm,bottom=2cm,left=3cm,right=3cm,marginparwidth=1.75cm]{geometry}

\usepackage{amsmath}
\usepackage{amsfonts}
\usepackage{amsthm}
\usepackage{graphicx}
\usepackage[colorlinks=true, allcolors=blue]{hyperref}
\usepackage{authblk}

\newtheorem{conjecture}{Conjecture}
\newtheorem{theorem}{Theorem}

\newtheorem{assumption}{Assumption}

\title{On Perles' configuration}

\author{Jozsef Solymosi}
\affil{University of British Columbia, Vancouver, Canada, and Obuda University, Budapest, Hungary}
\date{}                     

\begin{document}
\maketitle

\begin{abstract}
In the 1960s, Micha Perles constructed a point-line arrangement in the plane on nine points, which cannot be realized only by points with rational coordinates. 
Gr\"unbaum conjectured that Perles' construction is the smallest: any geometric arrangement on eight or fewer points if it is realizable with real coordinates in the plane, it is also realizable with rational coordinates. In this paper, we prove the conjecture.
\end{abstract}

\section{Introduction}
The problem we are going to consider has a long history in Euclidean geometry: Given a set of points, $P$ and a set of its 3-element subsets, ${\cal{S}}=\{S_1,\ldots,S_k\}$ where $S_i\in P\times P\times P$. Is there an arrangement of points in the Euclidean plane such that precisely the triples listed in $\cal{S}$ are collinear? Here, we are interested in configurations that can be realized with real but not rational coordinates. 

The oldest known construction for arrangements with no realization with rational coordinates only comes from von Staudt's work on the ``algebra of throws'' \cite{vonSt}. MacLane applied von Staudt's projective arithmetics to give an 11-point construction corresponding to $\sqrt{2}$ in \cite{MacLane} (fig. \ref{Mac}). 

\begin{figure}[h]
\centering
\includegraphics[scale=.6]{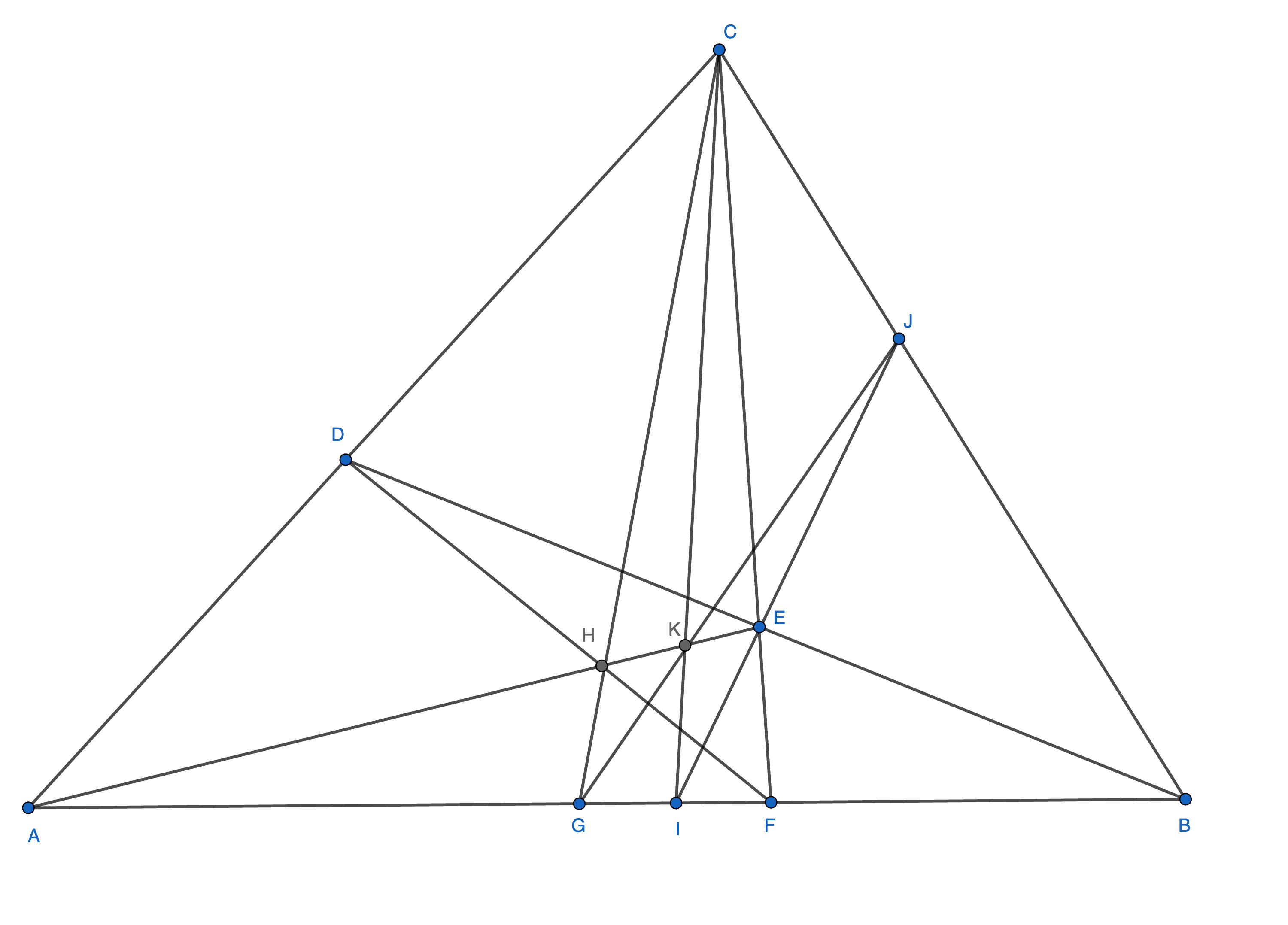}
\caption{This arrangement can not be realized with points from the integer grid}
\label{Mac}
\end{figure}

In the 60s, Perles gave an arrangement of nine points and nine lines such that no isomorphic arrangement exists where all points are rational (fig. \ref{Per}). 
The same configuration appears in a different context in \cite{ISS} as ``type 11''.
For the complete proof that Perles's construction can't be realized with rational coordinates only and more references on the subject, see Ziegler's excellent article in the Mathematical Intelligencer \cite{Zi}. Gr\"unbaum conjectured that Perles' configuration is the smallest; every arrangement on eight points has a rational realization \cite{BG}, \cite{Zi}. 

\begin{figure}[h]
\centering
\includegraphics[scale=.7]{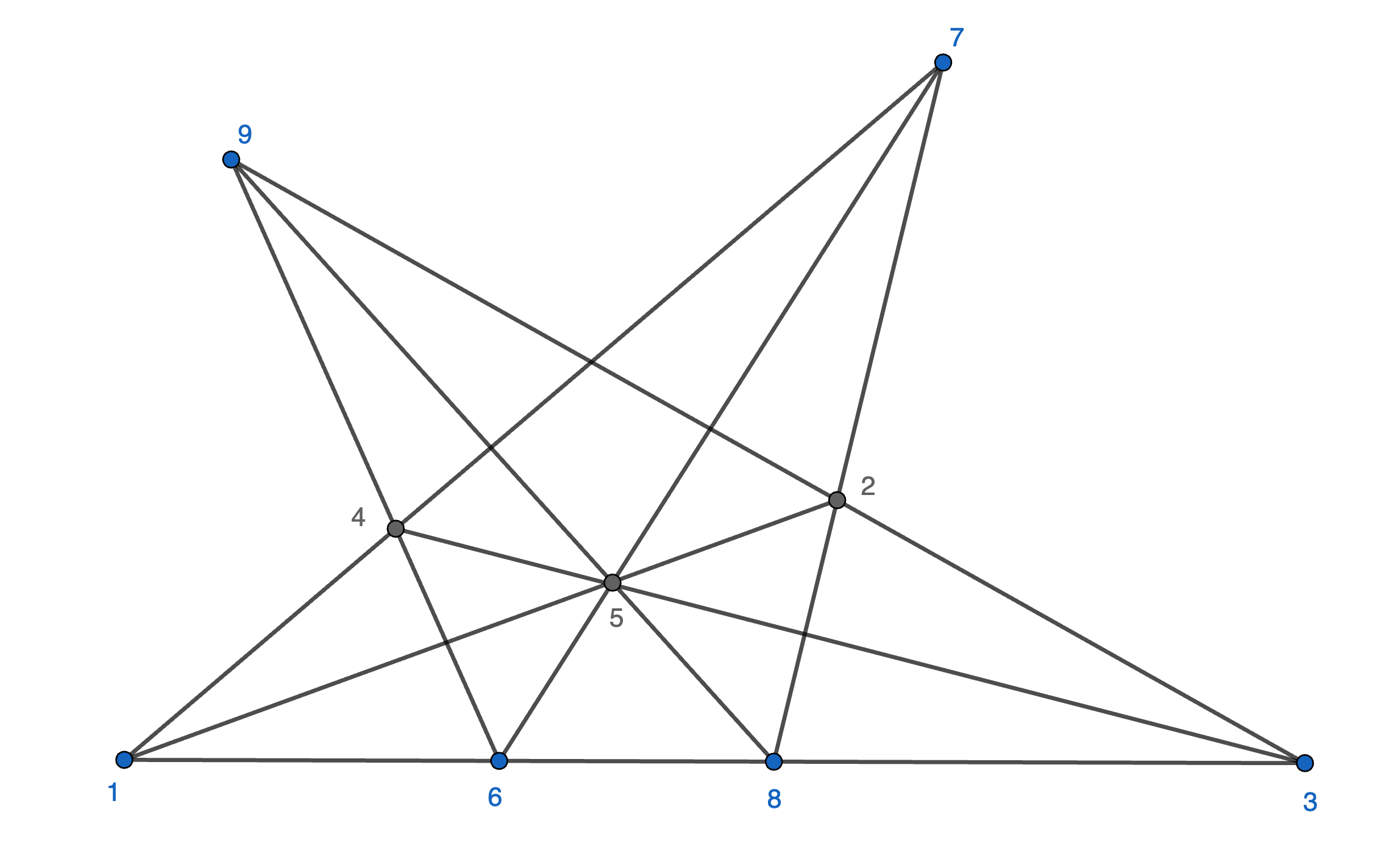}
\caption{Perles configuration. The labelling follows the list in \cite{ISS}}
\label{Per}
\end{figure}

\medskip
Arrangements with no four points collinear were also considered as early as the 19th century. Special attention was given to the $n_3$ {\em configurations}, arrangements where all the $n$ points have exactly $3$ collinear triples incident \footnote{We say a point $p$ is \textit{incident} to a line $\ell$ if $p\in\ell$, i.e. $p$ is a point of the line. Similarly, a point $p$ is {\em incident} to a collinear point set $S$ if $p\in S$.} to them. 
An incidence structure with a realization in the Euclidean plane is {\em rational} if it has a realization using points with rational coordinates.
The $n_3$ configurations were listed for small $n$ values, and Strumfeld and White proved the following theorem in \cite{SW}.

\begin{theorem}\label{n3_thm}
All $n_3$ configurations are rational for $9\leq n\leq 12$.
\end{theorem}

The cases $10_3$ and $9_3$ were proved earlier. Strumfeld and White used an incomplete list of configurations, as Gropp noted in \cite{Gr}. 
Gropp's paper gives a complete review of the realizations of configurations.
Gr\"unbaum stated the following conjecture in \cite{Gr2}:

\begin{conjecture}[Gr\"unbaum]
    Any $n_3$ configuration with a real realization also has a rational realization.
\end{conjecture}

The case $13_3$ was recently proved by Kocay \cite{Ko}, but the general case is still open.
\medskip

\section{Gr\"unbaum's Conjecture}

Now, we state our main result.

\begin{theorem}[Gr\"unbaum's Conjecture]
Any arrangement of $n\leq 8$ points is either not realizable with real coordinates or also has a rational realization.
\end{theorem}
\begin{proof} We consider four different cases. We will use the result of Kelly and Moser from \cite{KM}. Csima and Sawyer later improved this estimate in \cite{CsS}; however, we only need the bound below for small point sets.

\begin{theorem}[Kelly-Moser]
    In the Euclidean plane, among $n$ points, there are at least $3n/7$ ordinary lines (lines containing two points only) unless the $n$ points are collinear. 
\end{theorem}

Now, we state three simple assumptions about the arrangement. If they fail, we can reduce the number of points by analyzing the smaller arrangement.
The first one is quite obvious:

\begin{assumption}

Analyzing the various cases, we can assume that every point is incident to at least two collinear triples. Otherwise, we can remove the point. If the smaller configuration has a rational realization, we can choose a rational point on the line of the two remaining points of the original triple. We have infinitely many choices for points with rational coordinates, so we can choose a point without creating further collinearity not listed in the arrangement. 

\end{assumption}
One might argue we could assume every point is incident to at least three since if the arrangement without a point, $p$, incident to two triples, has a rational realization, then the lines of the triples through $p$ intersect in a point with rational coordinates. But in this case, we still have to guarantee that this $p$ is a new point of the arrangement and it won't create another collinear triple that is not on the list.

\begin{assumption}\label{ass2}
We can also assume that every line of collinear triples contains at least two points incident to at least three triples. Otherwise, we can remove the line and the points it is incident to, except the point, $p$, incident to at least three triples. If this smaller configuration has a rational realization, adding a line incident to $p$ with a rational slope gives a rational realization of the original configuration.

\end{assumption}


\medskip
Switching to the projective plane makes it easier to examine point-line arrangements in the real plane (two-dimensional Euclidean space). In this plane, we associate the point $(a,b)$ with a set of points in 3D homogeneous coordinates $(wa,wb,w)$ for all reals $w$. A point in the projective plane is rational if all three coordinates are rational. A rational arrangement ( point-line incidence structure) might contain points in the line at infinity ($w=0$). This is not an issue for us since such arrangements can be easily transformed into a rational arrangement in the Euclidean plane: Given a finite point set in the projective plane, find a regular line $ax+by+c=0$ with $a,b,c\in \mathbb{Q}$, which avoids all points in the arrangement. Apply the projective transformation
\[
T =
\begin{bmatrix}
1 & 0 & 0 \\
0 & 1 & 0 \\
a & b & c
\end{bmatrix}.
\]

This transformation moves all points of the arrangement to points where $w\neq 0$ (i.e. points of the Euclidean plane). $T$ also keeps rationality and preserves collinearity.

\begin{assumption}
If an arrangement has a real realization, then with a projective transformation, we can move three collinear points $a,b,c$ and a point outside their line, $d$ into points $a',b',c',d'$, with coordinates $a'=(1,0,0),b'=(0,1,0),c'=(1,1,0)$ and $ d'=(0,0,1)$. In more general terms, any arrangement of four points can be transformed into another set of four points if there is a bijection between the two point sets, preserving collinearity. 
\end{assumption}


\begin{itemize}
    \item First case: $\textbf{n=8}$ and no four points are collinear. Eight points define $\binom{8}{2}=28$ point-pairs where at least $\lceil 3\cdot8/7\rceil=4$ pairs determine ordinary lines by the Kelly-Moser bound. The remaining 24 pairs can span at most eight collinear triples. At least six collinear triples are needed to guarantee that every point is incident to at least two collinear triples. No point is incident to four or more triples since that would require at least nine points.
    \begin{itemize}
    
     \item If there are \textbf{eight collinear triples}, then every point has three triples incident to it. It is the unique $8_3$ configuration, which is realizable over the complex numbers but not over the reals.
    
     \item If there are \textbf{seven collinear triples}, then every triple contributes to three incidences, which gives 21 incidences. If there are $a$ points incident to three triples and $8-a$ to two triples, then $a*3+2(8-a)=21$, which has the $a=5$ solution. 
     
     The five degree-three\footnote{If a point is incident to $d$ collinear point sets (from $d$ distinct lines), we say its {\em degree} is $d$.} points span exactly one collinear triple.
     To see that, note that another way to calculate $a$ is using the number of triples where all three points are incident to three triples. If we denote the number of such triples by $x$, then $(3x+2(7-x))/3=a$, showing that there is one such triple. (Here, we used that every triple is incident to at least two other triples (Assumption \ref{ass2})).
     
     The points of this triple are denoted by $p_1,p_2,p_3$, and the remaining two degree-three points are denoted by $p_4,p_5$. Use a projective transformation sending $p_1,p_2,p_3,p_4$ to $$p'_1=(1,0,0),p'_2=(0,1,0),p'_3=(1,1,0),p'_4=(0,0,1).$$ 

     Then, $p'_5$ is a point outside the $x,y$-axes and the $y=x$ line. In a small neighbourhood of $p'_5$, let's select a rational point so that moving $p'_5$ there and leaving the other four degree-three vertices unchanged will keep the arrangement of collinear triples unchanged. Now, the remaining degree-two vertices have rational coordinates as well.

    \item If there are \textbf{six or fewer collinear triples}, then $a*3+2(8-a)\leq 18$, which means $a\leq 2$. But in this case, the two degree-three points are in at most one triple, so there are triples with less than two degree-three points, contradicting Assumption \ref{ass2}.
    \end{itemize}   

    \item Second case: If $\textbf{n=8}$ and there is a collinear quadruple. By Assumption \ref{ass2}, at least six collinear triples are incident to the collinear quadruple points. The four collinear points are denoted by $p_1,p_2,p_3,p_4,$ where $p_1$ and $p_2$ have degree three. Use a projective transformation sending $p_1,p_2,p_3,p_4$ to $$p'_1=(1,0,0),p'_2=(0,1,0),p'_3=(1,1,0),p'_4=(1,-1,0).$$ 

    There are four points left outside the line at infinity. The six collinear triples are each incident to two of them. All pairs are part of a collinear triple. Two triples are on horizontal and two are on vertical lines. The remaining two lines have slope one and slope negative one. This is a unique arrangement where the four points are vertices of an axis-parallel square. Its coordinates can be selected to be rational, so this arrangement has a rational representation.

    \item Third case: $\textbf{n= 7.}$ In this case, we don't have to consider collinear quadruples since it would require six more collinear triples with at least four more points, as we saw in the previous case. Suppose there are seven collinear triples, then $a*3+2(7-a)=21$, which means that all points are incident to three triples. It is the Fano plane, which has no real realization. 
    
    For six collinear triples, we get $a=4$. The four degree-three points span no collinear triple since that would require seven collinear triples. 
    If the arrangement has a real realization, then let's send the four degree-three points to the vertices of the unit square, $(0,0,1),(0,1,1),(1,0,1),(1,1,1)$. Then, the remaining points have degree two. By Assumption \ref{ass2}, every collinear triple has two points out of the four, i.e. they are on the lines of the sides and diagonals of the square. The unique positioning of the remaining three points are $(1/2,1/2,1),(1,0,0),(0,1,0)$, which gives a rational realization.

    \item Fourth case: ${n\leq 6.}$ There are four or fewer points with degree three, so the previous arguments work here.

\end{itemize}
\end{proof}

\section{An application}
Planar configurations, which are not realizable over the rationals, were the building blocks for nonrational polytopes. We refer the interested reader to the essential book on the subject by J\"urgen Richter-Gebert \cite{RGJ}. Here, we show another application, refuting two conjectures in discrete geometry.
Depending on the underlying field, there are different estimates of the maximum number of incidences. We count incidences as the number of point-line pairs in the arrangement where the point is on the line. Here, we are interested in arrangements in $\mathbb{R}^2$. 
For the sake of simplicity, we only consider arrangements where the number of points and lines is the same, $n$. By the Szemer\'edi-Trotter theorem  \cite{Sz-T}, the number of incidences is $O(n^{4/3})$. The implied constant is at most $2.5$, as shown in \cite{Pach+}. The Szemer\'edi-Trotter theorem is sharp up to the constant multiplier. There are examples of point-line arrangements of $n$ lines and $n$ points with $cn^{4/3}$ incidences. 
Points of the integer grid give the classic example of rich point-line incidences. Let $P=\lfloor\sqrt{n}\rfloor \times\lfloor\sqrt{n}\rfloor$ and $L$ is the set of lines containing at least $n^{1/3}$ points of $P$. In this arrangement, the number of lines is linear in $n,$ $cn$ for a universal constant $c>0$, and the number of incidences is $\Omega(n^{4/3})$. 
We refer to a recent paper \cite{BST} for details on this construction and the best-known bound. Integer grids are not the only known examples. There are other grid-like constructions that provide suitable pointsets for many incidences, as shown in \cite{Guth} and \cite{Cur}. 

\medskip
Let $P$ and $P_0$ be two sets of points and $L$ and $L_0$ be two sets of lines in the plane. We say that the pairs $(P,L)$
 and $(P_0,L_0)$ are isomorphic if there are bijections $f: P\rightarrow P_0$ and $g: L\rightarrow L_0$ such that $p\in P$ is incident to $\ell \in L$ iff 
 $f(p)\in P_0$ is incident to $g(\ell) \in L_0$. The following conjecture was stated in \cite{Soly_1}.

\begin{conjecture}\label{conj1}
For any set of points $P_0$
 and for any set of lines $L_0$, the maximum number of incidences between $n$ points and $n$ lines in the plane containing no subconfiguration isomorphic to $(P_0,L_0)$
 is $o(n^{4/3})$.
\end{conjecture}

The conjecture was proved in \cite{Soly_1} for the particular case when $P_0$ is a set of $k$ points with no three collinear, and the set of lines is given by the $\binom{k}{2}$ lines spanned by the points.

Mirzaei and Suk considered a grid-like forbidden configuration in \cite{MiSu} for which they gave a $O(n^{4/3-c})$-type incidence bound. They stated the following more ambitious conjecture.

\begin{conjecture}\label{conj2}
For any set of points $P_0$
 and for any set of lines $L_0$, there is a $c>0$ such that the maximum number of incidences between $n$ points and $n$ lines in the plane containing no subconfiguration isomorphic to $(P_0,L_0)$
 is $O(n^{4/3-c})$.
\end{conjecture}

In this form, the above conjectures don't hold. Any point-line arrangement with no realization using points with rational coordinates only will provide a counterexample for Conjectures \ref{conj1} and \ref{conj2}. 
Indeed, as mentioned earlier, the points of the $\lfloor\sqrt{n}\rfloor \times\lfloor\sqrt{n}\rfloor$ integer grid give an example for arrangements with many incidences, and every substructure there has points with integer coordinates. 

Very recently, Balko and Frankl also noticed that arrangements requiring irrational coordinates provide counterexamples to the above conjectures \cite{BF}. 

It is unclear what the characterization of unavoidable substructures in dense point-line arrangements would be. The only known dense arrangements are on points of a large Cartesian product. Their possible structure was analyzed in \cite{BST, Cur, Guth}. However, it is not known if there are arrangements that are very different from the examples described in these works.

\section{Other configurations}
In the previous examples, all arrangements that were not representable by rational coordinates only had at least four points on a line. 
As noted earlier, all $n_3$ configurations are conjectured to be realizable by rational coordinates. The general case is still open.
Noam Elkies sent me a configuration on ten points, not four on a line, which can not be realized using rational coordinates only \cite{El}. We present Elkies' construction in the appendix.

\section{Acknowledgements}
The author thanks Noam Elkies for providing references and sharing his construction and the anonymous referee for valuable suggestions. This research was partly supported by an NSERC Discovery grant and OTKA K grant no. 133819.

\section{Appendix}

We describe Elkies' construction. It is based on the properties of the elliptic curve, 
\[
 y^2 = (x-1) \cdot (x^2-2\cdot a-3) + ((a+1)\cdot x+1)^2/4
 \] 
but we stick to a geometric description similar to Perles' construction.

\medskip
\noindent
In the following construction, we use the labelling in Fig. \ref{calc}. 
\begin{enumerate}
    \item Select the vertices of an axis parallel square. The labels are $(1),(10),(4),(7).$
    \item On the vertical sides, select two points such that their line is horizontal (parallel to $(1),(10)$). The labels are $(3),(8).$
    \item The two lines determined by the point-pairs $(1),(8)$ and $(10),(3)$ are denoted by $\ell_1$ and $\ell_2$ resp.
    \item We will need four more lines for the construction. The first, $\ell_3$ is a horizontal line. Its intersection with $\ell_1$ is a new point labelled by $(2)$, and with $\ell_2$ is labelled by $(9).$
    \item The two lines, defined by the points $(9),(4)$ and $(2),(3)$, are denoted by $\ell_4$ and $\ell_5$.
    \item The last line, $\ell_6$ is the diagonal of the square, connecting $(1)$ and $(4)$. Its intersection point with $\ell_4$ is denoted by $(6')$ and its intersection with $\ell_5$ is denoted by $(6")$.
    \item Move $\ell_3$, keeping it parallel to $(1),(10)$ until the two points $(6')$ and $(6")$ coincide. This point is labelled by $(6).$ 
    \item Let's add to the construction the point at infinity in the line $\ell_3$ and label it by $(5)$.
\end{enumerate}

\noindent
The collinear triples are $((1),(5),(10)), ((3),(5),(8)), ((4),(5),(7)), ((2),(5),(9)), ((1),(3),(7)),\\ ((4),(8),(10)), ((1),(2),(8)), ((3),(9),(10)),
 ((1),(4),(6)), ((2),(3),(6)),$ and $ ((6),(7),(9)).$

 \medskip

\begin{figure}[h]
\centering
\includegraphics[scale=.40]{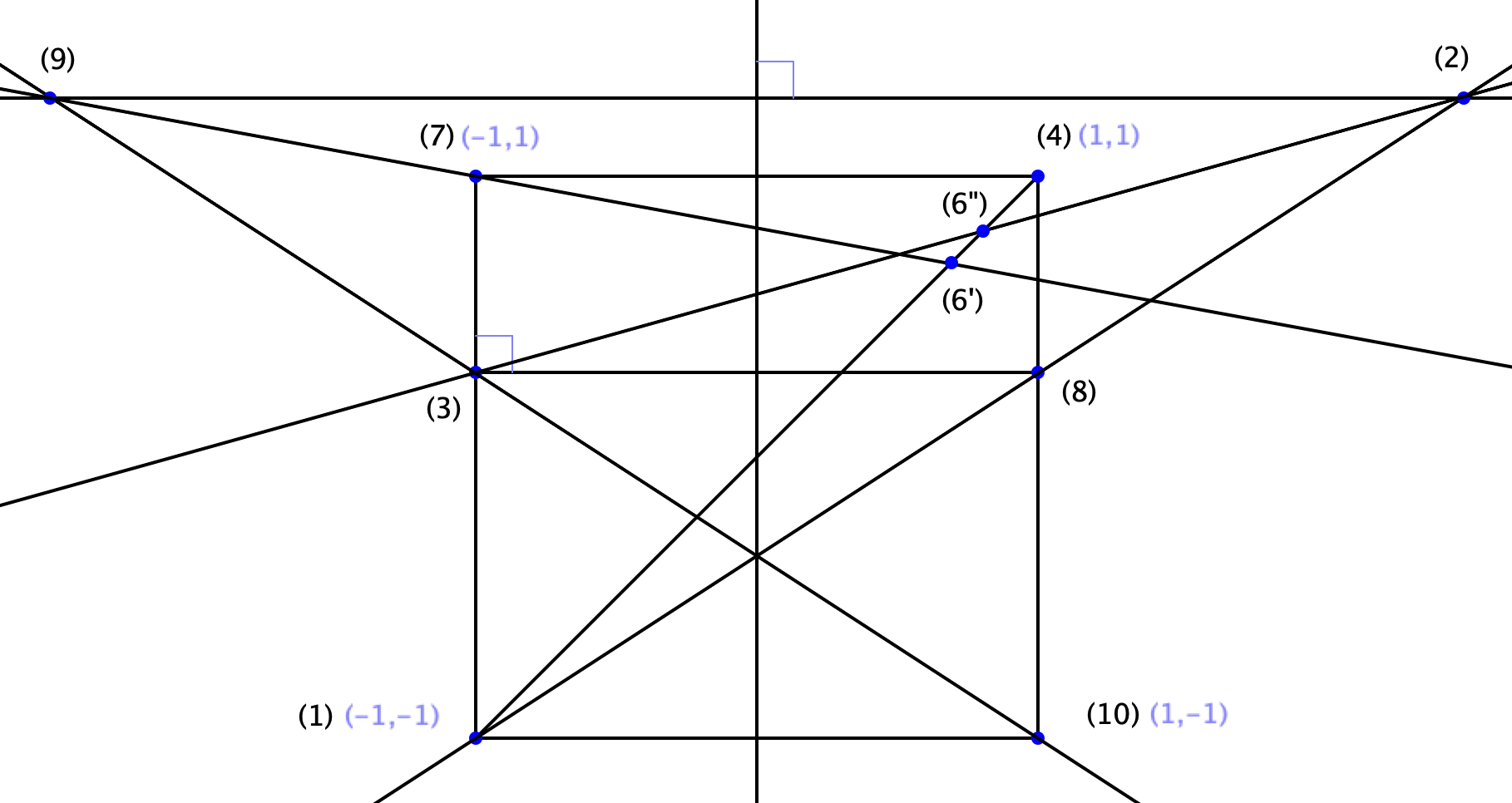}
\caption{Points $(6')$ and $(6")$ should be the same. }
\label{calc}
\end{figure}

\medskip

Now, we sketch the calculations showing that any real realization of this arrangement contains a point with irrational coordinates.
Given a realization of the 10-point configuration above, let's apply a plane-to-plane projectivity, which maps the selected four points as 
$$(4)\rightarrow (1,1)\quad (10)\rightarrow (1,-1)\quad (7)\rightarrow (-1,1)\quad (1)\rightarrow (-1,-1)$$

Let's calculate point $(6)$ coordinates using two parameters. The first, denoted by $a$ is the $y$-coordinate of point $(8)$ and the other, denoted by $b$ is the $y$-coordinatate of point $(2)$. Now that the points are given, we deviate a bit from our previous notations of lines, and a line connecting points $(i)$ and $(j)$ is denoted by $\ell_{i,j}$. There are two ways to find the coordinates of point $(6)$ from the fixed four points and the two parameters. 
\begin{itemize}
    \item First calculation: Points $(8)$ and $(3)$ have the same $y$ coordinates  $a$. The intersection of $\ell_{1,8}$ and the $y=b$ line determines point $(2)$. 
    Its $x$-coordinate is $(2b-a+1)/(a+1)$. (Note that then the $x$-coordinate of $(9)$ is $-(2b-a+1)/(a+1)$ by the symmetry induced by the selection of the four points above)
    The intersection of $\ell_{2,3}$ and $\ell_{1,4}$ (which is the $x=y$ line) is point $(6)$. Its $x,y$ coordinates are 
    \[
    \frac{-a^2+3ab+a+b}{a^2-ab+a+b+2}.
    \]

    \item Second calculation: We have to find the intersections of $\ell_{7,9}$ and $\ell_{1,4}$.
    Both of its coordinates are 
    \[
    \frac{-(a-1)(b+1)}{a(b-3)+3b-1}.
    \]
   The two results should be the same. 
   \[
   \frac{-a^2+3ab+a+b}{a^2-ab+a+b+2}= \frac{-(a-1)(b+1)}{a(b-3)+3b-1}.
   \]

   After solving it for $b$, we have
    \[
    b=\frac{3a^2+1+(a+1)\sqrt{2}\sqrt{-a^3+a^2+a+1}}{a^2+6a+1}.
    \]

    Neither the polynomial under the square root nor in the denominator has rational roots. For a possible rational representation, both $a$ and $b$ should be rational, which would require a rational solution of the 

    \[
    2y^2=x^3+x^2-x+1
    \]

    equation. The $a=\pm 1$ value is not a solution in our configuration since two pairs of points would be the same. This elliptic curve has no other finite rational points. It is birationally equivalent to the $y^2+y=x^3-x^2$ elliptic curve. This is a model for the modular curve $X_1(11)$ listed in \cite{elliptic}.

\end{itemize}

\end{document}